\documentclass[11pt, final]{amsproc}

\usepackage{amsthm}
\usepackage{enumitem}
\usepackage{amssymb,amsmath,latexsym}
\usepackage{lineno,hyperref}
\usepackage{xcolor}

\newtheorem{theorem}{Theorem}[section]

\newtheorem{corollary}[theorem]{Corollary}
\newtheorem{definition}[theorem]{Definition}

\newtheorem{lemma}[theorem]{Lemma}
\newtheorem{proposition}[theorem]{Proposition}
\newtheorem{problem}{Problem}
\theoremstyle{remark}
\newtheorem{remark}[theorem]{Remark}
\newtheorem{algorithm}{Algorithm}

\def\F{\mathcal{F}}
\def\cl{\overline}
\def\Fk{\mathcal{F}_k}
\def\EI{\mathcal{E}_\infty}

\def\t{\mathfrak{t}}
\def\e{\varepsilon}
\def\R{\mathbb{R}}

\begin{document}

\title{Chebyshev approximation of exponential data}

\author[M.\ Rodr\'iguez-Arias, J.\ Cabello, J.A.F.\ Torvisco]
{Mariano Rodr\'iguez-Arias Fern\'andez, Javier Cabello S\'anchez, \\ 
Juan Antonio Fern\'andez Torvisco}
\address{Departamento de Matem\'aticas and Instituto de Matemáticas, 
Universidad de Extremadura. Avda. de Elvas s/n, 06006 Badajoz; Spain. \\
arias@unex.es; coco@unex.es; jfernandck@alumnos.unex.es}

\begin{abstract}
In this paper we present an algorithm to fit data via exponentials when the 
error is measured using the $\max$-norm. 
We prove the necesssary results to show that the algorithm will converge to 
the best approximation no matter the dataset. 
\end{abstract}



\maketitle

\section{Introduction}

Given some data $\t=(t_1,\ldots,t_n)\in\R^n,$ $T=(T_1,\ldots,T_n)\in\R^n$, in this paper we 
are going to show how to determine the coefficients $a,b,k\in\R$ that make the 
exponential $f(t)=a\exp(kt)+b$ minimize the error 
\begin{equation}
\EI=\|(T_1-f(t_1),\ldots,T_n-f(t_n))\|_\infty=\max\{|T_1-f(t_1)|,\ldots,|T_n-f(t_n)|\}. 
\end{equation}

It is usual to search for the best fitting with least squares due to the simplicity 
of the procedure: as the usual norm in $\R^n$ is differentiable, looking for 
the minimum of some distance is equivalent to find the point where every 
partial derivative is 0. In spite of this, it may seem more natural to minimize 
the width of a band that contains every point---and this is exactly what 
Chebyshev approximation (approximation where the errors are measured with 
the $\max$-norm) achieves (see, e.g.,\cite{Descloux, Quesada}). 

One of the main advantages of this kind of approximation is that, eventually, 
the problem reduces to the study of the case $n=4$; the main issue is 
to determine how to reduce the problem. Once this is done, we just need to find 
the only $(a,b,k)$ that will fulfil~(\ref{los4}). 


In order to ease the notation throughout the paper, for any map $f:\mathbb{R} 
\to \mathbb{R}$  and $v=(v_1,\ldots,v_n)\in\mathbb{R}^n$, $f(v)$ denotes  
$(f(v_1),\ldots,f(v_n))\in \mathbb{R}^n$ and, for every $i<j<l$, $v_{i,j,l}$ 
denotes $(v_i,v_j,v_l)$. Let $\overline{\F}$ denote $\{ ae^{kt}+b \text{ with  } 
a,b,k \in \mathbb{R} \}$ and $\overline{\Fk}$ denote $\{ ae^{kt}+b \text{ with  } 
a,b \in \mathbb{R} \}$ for certain $k \in \mathbb{R}$, then $\F^v$ denotes 
$\{ f(v) : f \in \overline{\F}  \}$ and $\Fk^v$ denotes $\{ f(v) :  f \in \overline{\Fk}  \}$.

In what follows, $T$ and $\t$ represent vectors in $\mathbb{R}^n$, being 
$\t=(t_1,\ldots,t_n)$ such that $t_i<t_{i+1},$ $ \forall\, i=1,\ldots,n-1$. Vectors 
$T$ and $\t$ are arbitrary but fixed throughout the paper. To ease even more the 
notation, we are going to denote $\F^{\t}$ as $\F$, and $\Fk^{\t}$ as $\Fk$.

So, what we are going to solve is the following: 

\begin{problem}\label{problem1}
Given $T$ and $\t$, to approximate $T$ with $\F$ as the family of approximants 
and the $\max$-norm as the approximation criteria.
\end{problem}
In this problem, $T$ represents the data, $\t$ the instants when the data were 
recorded, $\F$ the family where we intend to find the closest element to 
$T$, and  the $\max$-norm is the approximation criteria: the way we measure how 
close to $T$ each element of $\F$ is. With this in mind, we say that $f(\t)\in\F$ is a 
best approximation of $T$ in $\F$ if $\| T-f(\t) \|_\infty \leq \| T-g(\t) \|_\infty$, 
whenever $g(\t)\in \F$. We also say that $f(\t)$ is {\em the} best approximation 
if the inequality is strict for every $g(\t)\neq f(\t)$.

When we say that we solve Problem \ref{problem1}, we mean we can determine  
the expression of the best approximation. The same applies for other problems. 

We have not been able to directly solve Problem \ref{problem1}. We can directly 
solve, however, the following related problem: 

\begin{problem}\label{problem2}
Given $T$ and $\t$ and given $k\in \mathbb{R}$, to approximate $T$ with $\Fk$ as 
the family of approximants and the $\max$-norm as the approximation criteria. 
\end{problem}

The existence and uniqueness of the best approximation for Problem \ref{problem2} 
were proved in \cite{tac}, and conditions for the existence and uniqueness of the 
best approximation for Problem \ref{problem1} were established in \cite{tac_mdpi}. 
However, a method to find the best approximation was not presented for neither 
Problem. In this paper we provide some methods to solve Problems \ref{problem1} 
and \ref{problem2} when the best approximation exists.  

In Section \ref{s_fixed_k}, Problem \ref{problem2} is solved by calculating a 
function $f_k \in \overline{\F_k}$ such that $f_k(\t)$ is the best approximation of 
Problem \ref{problem2}. We present two different methods to solve such problem.

In Section \ref{s_every_k}, Problem \ref{problem1} is solved by having into 
account that $\F = \cup_{k \in \mathbb{R}} \F_k$. Note that the solution of 
Problem \ref{problem1} is also the solution of Problem \ref{problem2} for 
certain $k\in \mathbb{R}$. Being able to solve Problem \ref{problem2} for any 
$k$ allows us to approximate the $k$ that solves Problem \ref{problem1} since 
the error function $\EI$, defined as $\EI (k):= \| T-f_k(\t)  \|_\infty$, is 
quasiconvex, as proven in \cite[theorem 2.17]{tac}. 

In the last section we present a way to improve the algorithm in some cases. 
As this is not as robust as the algorithm itself, we have decided to deal with 
it in a separate section. 


\section{Best approximation for Problem \ref{problem2}} \label{s_fixed_k}
The aim of this section is to present two methods to solve Problem \ref{problem2}. 
We will refer to the first method as {\em exhaustive method}; from a computational 
point of view, it has low efficiency---although for small datasets it is a good fit. 
The second one is a recursive method. 
The exhaustive method shows the straightforward way to solve 
Problem \ref{problem2} and hints the necessity to approach the problem 
from a different perspective; the second method does this.

\subsection{Exhaustive method (solving Problem \ref{problem2} by 
solving several minor problems)}\label{ss_exhaustive}
This method determines a finite family in $\Fk$ containing the solution of 
Problem \ref{problem2}. Comparing the error of each element from such finite 
family and choosing the element which provides the minimum error leads us to 
the solution. Before we go any further, let us define (and solve) a new problem:
\begin{problem}\label{problem3}
Given $T_{i,j,l}$ and $\t_{i,j,l}$, to approximate $T_{i,j,l}$ using 
$\Fk^{\t_{i,j,l}}$ as the family of approximants and the $\max$-norm 
as the approximation criteria.
\end{problem}

Let us solve Problem \ref{problem2} with the exhaustive method. If $k=0$, 
the solution is trivial. Note that every couple of real numbers $(a,b)$ fulfilling 
$a+b=\frac{\max(T)+\min(T)}{2}$ solves Problem \ref{problem2}. Indeed: the function 
$h$ defined as $h(t)=a+b\in \overline{\F_0}$, makes $h(\t)\in \F_0$ the best 
approximation of $T$ since $h$ is equidistant from the maximum and from the 
minimum---we have written $h(t)=a+b$ for the analogy, say, $h(t)=b+a\exp(0t)$. 
From now on we will denote as $h(\t)$ the best approximation of $T$ in $\F_0$, 
or equivalent, $h(\t)$ will denote the solution of Problem \ref{problem2} for $k=0$.

Let us solve Problem \ref{problem2} for $k \neq 0$. According to \cite[proposition 5]{tac_mdpi}, 
for any $k$, the coefficients $a$ and $b$ determining the solution of Problem \ref{problem2} 
coincide with the coefficients of the solution of Problem \ref{problem3} for certain $i<j<l$ 
in $\{1,\ldots,n \}$. Coefficients $a$ and $b$ of the solution of Problem \ref{problem3} are 
computed using \cite[lemma 1]{tac_mdpi}; such coefficients are: 
$a=\frac{T_i-T_l}{e^{kt_i}-e^{kt_l}}$ and $b=\frac{1}{2}(T_i-ae^{kt_i}+T_j-ae^{kt_j})$. 

To solve Problem \ref{problem2}, we are going to solve Problem \ref{problem3} 
for every $i<j<l$ in $\{1,\ldots,n \}$. Let us define the (finite) subset 
$A \subset \Fk$ containing the solution of Problem \ref{problem2}: 
$A= \left\{ a_{il}e^{k\t}+b_{ij}: i<j<l \right\}, $
where
$$
a_{il}=\frac{T_i-T_l}{e^{kt_i}-e^{kt_l}}, \; 
b_{ij}=\frac{1}{2}(T_i-ae^{kt_i}+T_j-ae^{kt_j}).
$$

Comparing the errors obtained by fitting with each element of $A$ and choosing 
the element with the minimum error, we solve Problem \ref{problem2}.

\begin{algorithm}[Exhaustive search of the best approximation]\label{algorithm_exhaustive}
\hfill \break
{\tt Procedure exhaustive search} ($T$, $\t$ and $k \neq 0$.) 
\begin{enumerate}
\item \textbf{For every} $i<j<l$, $a_{il}:=\frac{T_i-T_l}{e^{kt_i}-e^{kt_l}}$, \\
$b_{ij}:=\frac{1}{2}(T_i-a_{il}e^{kt_i}+T_j-a_{il}e^{kt_j})$.

\item \textbf{If} $\| T-(a_{\hat{i}\hat{l}}e^{k\t}+b_{\hat{i}\hat{j}}) 
\|_\infty = \min \{\| T-(a_{il}e^{k\t}+b_{ij}) \|_\infty  \,:\, i<j<l  \}$ \\
\textbf{then} $a:=a_{\hat{i}\hat{l}}$, $b:=b_{\hat{i}\hat{j}}$.
\end{enumerate}
\end{algorithm}

This algorithm is very simple (conceptually) and gives the exact best approximation, 
but doing calculations this way may have a serious drawback: its processing 
time. For the coefficient $a$, the number of possibilities that need to be computed is 
$\sum_{i=1}^{n-2} i = \frac{(n-1)(n-2)}{2}$, and, for $b$, the 
number of possibilities is $\sum_{i=1}^{n-2} \frac{i(i+1)}{2} = 
\frac{1}{4}(n-1)(n-2)\frac{2(n-2)+4}{3}$. These numbers hint the 
amount of computation needed.

Algorithm \ref{algorithm_exhaustive} can be eased by taking into account 
\cite[lemma 6]{tac_mdpi}, which allows to directly determine $b$ once 
the coefficient $a$ has been calculated. 


\subsection{Successive remainders method (no minor problems needed) }\label{ss_remainders}
To overcome this issue, we present a different 
method to solve Problem \ref{problem2}.

As the case $k=0$ is fully solved, being $h(\t)$ the solution, we are going to 
focus on the case $k \neq 0$. Please take into account that maximum or minimum 
means {\em global} maximum or minimum.

\begin{definition}
We say $T$ is alternated whenever it has alternating maxima and minima, 
that is, a minimum is reached between two maxima or vice versa. In a more detailed 
way, one of the following situations happens: 
\begin{itemize}
\item $T_{M_1}=T_{M_2}=\max T$, $T_{m_1}=\min T$ and ${M_1}<{m_1}<{M_2}$ 
\item $T_{M_1}=\max T$, $T_{m_1}=T_{m_2}=\min T$ and ${m_1}<{M_1}<{m_2}$.
\end{itemize}
\end{definition}

\begin{proposition}\label{p_min_max_min_or_vice_versa}
Consider an alternated $T$. Then $h(\t)$ is the solution of Problem $\ref{problem2}$.
\end{proposition}
\begin{proof}
Assuming that alternation in $T$ is a minimum between two maxima (proof for the 
other alternation is analogous), let $T_{M_1}$ and $T_{M_2}$ denote those maxima 
and $T_{m_1}$ the minimum, being ${M_1}<{m_1}<{M_2}$. Let $f$ be a non-constant 
exponential function and suppose that $f(\t)$ fits $T_{M_1}$ and $T_{M_2}$ better than  
$h(\t)$. This means that $|T_{M_1}-f(t_{M_1})|\leq T_{M_1}-h(t_{M_1})$ and 
$|T_{M_2}-f(t_{M_2})|\leq T_{M_2}-h(t_{M_2})$, and therefore 
$f(t_{M_1})\geq h(t_{M_1})=h(t_{M_2})$. Since $f$ is strictly monotonic, $f\geq h$  in 
$[t_{M_1},t_{M_2}]$. As $t_{m_1}\in (t_{M_1},t_{M_2})$, $f(t_{m_1})>h(t_{m_1})>T_{m_1}$ 
(this last inequality is true because of the definition of $h$), and 
$f(t_{m_1})-T_{m_1}>h(t_{m_1})-T_{m_1}$. As we are dealing with the max-norm, 
the last inequality makes $\|T-h(\t)\|_\infty<\|T-f(\t)\|_\infty$, and therefore 
$f(\t)$ is a worse approximation of $T$ than $h$.
\end{proof}

Proposition \ref{p_min_max_min_or_vice_versa} shows that an alternated $T$ does 
not deserve a closer look since its solution is trivial. Therefore, we are going 
to focus on the situations related with non-alternated $T$. Please observe 
that if $T$ is non-alternated we can face 4 possible cases, namely:
 \begin{enumerate}[label=Case \arabic*.,leftmargin=3\parindent]
\item $k<0$ and the maxima of $T$ are attained before its minima, so $a_k>0$.
\item $k<0$ and the minima of $T$ are attained before its maxima, so $a_k<0$.
\item $k>0$ and the maxima of $T$ are attained before its minima, so $a_k>0$.
\item $k>0$ and the minima of $T$ are attained before its maxima, so $a_k<0$.
\end{enumerate}
\begin{remark}\label{r_every_case_is_ok}
We can assume without loss of generality that Problem \ref{problem2} is on 
case 1; otherwise it can be reduced to that particular case.
\end{remark}
Indeed, cases 2, 3 and 4 can be reduced to case 1 through some symmetries: case 2 
reduces to case 1 with an $X$-axis symmetry, case 3 reduces to case 1 with an 
$X$-axis and a $Y$-axis symmetry and case 4 reduces to case 1 with a $Y$-axis 
symmetry. Those symmetries will change either the sign of $k$ ($Y$-axis symmetries) 
or the sign of $a$ and $b$ ($X$-axis symmetries) in the solution we seek. It is 
enough to revert back the changes once we find the solution for the symmetric case.
In spite of this, we need to determine the case that we are dealing with. 
For this, we can use the classical Remez Algorithm (see, e.g.,~\cite{ReemtsenRemez, 
Saramaki, Remez}) to find the line $f(t)=d+ct$ that best approximates $T$. Once 
we have determined the best linear approximation $f(t)=d+ct$, we have the following: 
\begin{lemma}
$T$ is in case 1 if and only if $d<0$ and there exist 
$i<j<l$ such that 
\begin{equation}\label{ijl}
T_i-f(t_i)=-(T_j-f(t_j))=T_l-f(t_l)=\|T-f(\t)\|_\infty,
\end{equation}
see~\rm{\cite[definition~2]{tac_mdpi}. }
\end{lemma}

Since Problem \ref{problem2} is in case 1, $T$ has its maxima lying before its 
minima, so $T$ must somehow have a decaying aspect. It seems clear that an 
exponential function $I_h \equiv ae^{kt}+b$, with $k<0$, should have $a>0$ to 
replicate that decaying aspect of $T$. In that way, some $I_h(\t)\in \Fk$ could 
provide a better approximation than $h(\t)\in \Fk$. Let us construct $I_h$.

Let $t_m$ be the first coordinate where $T$ attains its minimum and $t_M$ the 
last coordinate where $T$ attains its maximum (observe that $m>M$). 
This implies $T_m<T_i, \; \forall\, i<m$ and $T_i<T_M, \; \forall\, i>M$.

Consider any $T_i$ and $T_j$. For this $k$, there exist a unique exponential 
function $a_{ij}e^{kt}+b_{ij}$ interpolating $(t_i,T_i)$ and $(t_j,T_j)$. It 
is immediate that the value $a_{ij}$ is given by the expression
\begin{equation}\label{eq_aij}
a_{ij}=\frac{T_i-T_j}{e^{kt_i}-e^{kt_j}}.
\end{equation}

Let us consider the $a_{im}$ coefficients of exponentials interpolating $(t_m,T_m)$ 
and $(t_i,T_i)$ for every $i \neq m$. Please note that, by \eqref{eq_aij}, $a_{im} 
\leq 0,\;$ $\forall\, i > m$ and $a_{im} > 0,\;$ $\forall\, i < m$. Remind we are looking 
for $a$ to be positive in $I_h$, so we are going to consider just $a_{im}$ $\forall\, i < m$.

In a similar way, let us consider the $a_{jM}$ coefficients of exponentials 
interpolating $(t_M,T_M)$ and $(t_j,T_j)$ for $j \neq M$. This time, $a_{jM} 
\leq 0,$ $\forall\, j < M$ and $a_{jM} > 0,$ $\forall\, j > M$, so, for the same 
reason, we are going to consider just $a_{jM},$ $\forall\, j>M$.

In order to determine $I_h$ we are going to consider the coefficients 
\begin{eqnarray}\label{eqnIh}
a:=&\min\{ a_{im}, a_{jM}, \text{ with } i<m \text{ and } j>M \}, \\ 
b_m:=&T_m-ae^{kt_m} \text{\quad and }\quad b_M:=T_M-ae^{kt_M}. \nonumber
\end{eqnarray}
Note that $a>0$. Also, 
observe that $(t_m,T_m)$ and $(t_M,T_M)$ are, respectively, on the 
graph of $I_h^m \equiv ae^{kt}+b_m$ and $I_h^M \equiv ae^{kt}+b_M$.

The following Theorem can be considered as the cornerstone of this section. 
Please note that we are not claiming such element to be the best 
approximation of $T$, that is, the solution of Problem \ref{problem2}.

\begin{theorem}[Constructive method]\label{t_exp_constructive}
Let Problem $\ref{problem2}$ be on case $1$. Then, the function given by 
$I_h(t)=ae^{kt}+(b_m+b_M)/2$ and whose coefficients are defined 
in~$(\ref{eqnIh})$ improves the approximation of $h(\t)$.
\end{theorem}
\begin{proof}

An instant to note that if $a=a_{mM}$, then $I_h^m=I_h^M$, and defining $I_h:=I_h^m$, 
we have that $I_h(\t)$ interpolates $T$, thus solving Problem \ref{problem2} and 
Problem \ref{problem1}. Therefore we assume $a \neq a_{mM}$.

Whenever $a \neq a_{mM}$, $T$ will remain between the bands constituted 
by $I_h^m$ and $I_h^M$, that is, $I_h^m (t_i) \leq T_i \leq I_h^M(t_i), $ 
$\forall\, i=1,\ldots,n$. Let us prove the first inequality: given 
$l \geq m$, inequality $I_h^m(t_l) \leq T_l$ is trivial. Given $l<m$, inequality 
$I_h^m(t_l) \leq T_l$ must also hold, otherwise $T_l<I_h^m(t_l)$ implies 
$a_{lm}<a$, see \eqref{eq_aij}, which contradicts the definition of $a$. 

The second part of the inequality is symmetric: given $l \leq M$, inequality 
$T_l \leq I_h^M(t_l)$ is trivial. Given $l>M$, inequality $T_l \leq I_h^M(t_l)$ 
must also hold, otherwise $I_h^M(t_l)<T_l$ implies $a_{lM}<a$, see 
\eqref{eq_aij}, which, again, contradicts the choice of $a$.

It is obvious that $I_h:=\frac{1}{2}(I_h^m+I_h^M)\in \overline{\F_k}$ is equidistant 
from $I_h^m$ and $I_h^M$. Reminding $(t_m,T_m)$ and $(t_M,T_M)$ are, respectively, 
on the graph of $I_h^m$ and $I_h^M$, the following equalities hold:
\begin{equation}\label{eq_max_norm}
\| T-I_h(\t) \|_\infty=T_M- I_h(t_M)=-(T_m-I_h(t_m)). 
\end{equation}

Note that $I_h \equiv ae^{kt}+b$, being $b=\frac{1}{2}(b_m+b_M)$. Once $I_h$ has 
been constructed, all is left to prove is that $I_h(\t)$ fits $T$ better than 
$h(\t)$, that is, $\| T-I_h(\t) \|_\infty<\| T-h(\t) \|_\infty$. The inequality 
is immediate because $\| T-I_h(\t) \|_\infty=T_M-I_h(t_M)$, $\| T-h(\t) 
\|_\infty=T_M-h(t_M)$ and $T_M-I_h(t_M)<T_M-h(t_M)$ as can be checked by
\begin{equation*}\label{eq_f_above_h}
\begin{split}
 I_h(t_M) & =ae^{kt_M}+b=ae^{kt_M}+\frac{1}{2}(b_m+b_M)\\
 & = ae^{kt_M}+\frac{1}{2}\big((T_m-ae^{kt_m})+(T_M-ae^{kt_M})\big) \\ 
 & = \frac{T_M+T_m}{2}+ \frac{a}{2} (e^{kt_M}-e^{kt_m}) = 
 h(t_M) + \frac{a}{2} (e^{kt_M}-e^{kt_m})\\
 & > h(t_M).
\end{split}
\end{equation*}
\end{proof}


As the last proof shows, see \eqref{eq_max_norm}, the $\max$-norm is 
attained in $t_m$ and $t_M$. Now we are going to prove that the  
$\max$-norm is going to be attained in at least another observation 
$(\tilde{t},\tilde{T})$. In more detail, 
\begin{corollary}\label{r_3_points} In the hypotheses of Theorem 
\ref{t_exp_constructive}, either 
\begin{equation}\label{eq_min_max_min}
\| T-I_h(\t) \|_\infty=T_M- I_h(t_M)=-(T_m-I_h(t_m))=-(\tilde{T}-I_h(\tilde{t})), 
\end{equation}
with $\tilde{t}<t_m$ or
\begin{equation}\label{eq_max_min_max}
\| T-I_h(\t) \|_\infty=T_M- I_h(t_M)=-(T_m-I_h(t_m))=\tilde{T}-I_h(\tilde{t}),
\end{equation}
with $\tilde{t}>t_M$.
Furthermore, any observation $(\tilde{t},\tilde{T})$ such that 
$\| T-I_h(\t) \|_\infty=-(\tilde{T}-I_h(\tilde{t}))$ fulfils $\tilde{t} \leq t_m$ 
and any observation $(\tilde{t},\tilde{T})$ such that 
$\| T-I_h(\t) \|_\infty=\tilde{T}-I_h(\tilde{t})$ fulfils $\tilde{t} \geq t_M$.
\end{corollary}
\begin{proof}
If $a=a_{im}$, then $I_h^m(t_m)=T_m$, $I_h^m(t_i)=T_i$, and, taking $\tilde{t}=t_i$, 
\eqref{eq_min_max_min} holds. If  $a=a_{jM}$, then $I_h^M(t_M)=T_M$, $I_h^M(t_j)=T_j$, 
and, taking $\tilde{t}=t_j$, \eqref{eq_max_min_max} holds.

As for the furthermore part, since $T_m$ is the first minimum and $I_h$ a strictly 
decreasing function, $-(T_m-I_h(t_m))>-(\tilde{T}-I_h(\tilde{t}))$ for $\tilde{t}>t_m$. 
An analogous reasoning give us $T_M- I_h(t_M)>-(\tilde{T}-I_h(\tilde{t}))$ 
for $\tilde{t}<t_M$, so we are done. 
\end{proof}

%


\begin{proposition}\label{c_M_m_unchanged}
Considering $T$ in case $1$ then, $T-I_h(\t)$ is either in case $1$ or alternated.
\end{proposition}
\begin{proof}
As stated in \eqref{eq_max_norm}, $\| T-I_h(\t) \|_\infty=T_M-I_h(t_M)=-(T_m-I_h(t_m))$. 
Therefore, $T-I_h(\t)$ has a maximum in $t_M$ and a minimum in $t_m$.

Attending to Corollary \ref{r_3_points}, when there is an instant $\tilde{t}$ where 
the norm is attained as $\| T-I_h(\t) \|_\infty=-(\tilde{T}-I_h(\tilde{t}))$, 
$\tilde{t}$ is located before $t_m$ (that means $(\tilde{t},\tilde{T})=(t_r,T_r)$ 
for certain $r<m$). On the same way, for any instant $\tilde{t}$ where the norm 
  is attained as $\| T-I_h(\t) \|_\infty=\tilde{T}-I_h(\tilde{t})$, $\tilde{t}$ 
  is located after $t_M$. Please remind that Corollary \ref{r_3_points} ensures 
  such instant $\tilde{t}$ exist.

Having into account that
\begin{itemize}
\item $t_m$ and $t_M$ are instants where a minimum and a maximum 
of $T-I_h(\t)$ are located, respectively, 
\item $t_M<t_m$, 
\item If a minimum of $T-I_h(\t)$ is attained in $\tilde{t}$ then $\tilde{t} \leq t_m$, 
\item If a maximum of $T-I_h(\t)$ is attained in $\tilde{t}$ then $\tilde{t} \geq t_M$,
\end{itemize}
we can conclude that, unless an alternation exists (minima and maxima mix), 
any maximum of $T-I_h(\t)$ lies before all minima. Also, it is clear that the 
parameter $a_k$ of the best approximation is $a_{ij}$ for some pair of indices, 
so $a_k$ is greater than the expression~(\ref{eqnIh}) and this implies that 
$T-I_h(\t)$ is in case 1.
\end{proof}

In Theorem \ref{t_exp_constructive} we have improved the approximation of $h(\t)$ 
with an exponential $I_h(\t)$. It is time to iterate this improvement. 
Please recall that $k<0$ is fixed. 

Suppose $T$ in case $1$ and $I_h$ is not its best approximation. 
Then $(\t,T-I_h(\t))$ is also in case $1$, so the function 
prescribed by Theorem \ref{t_exp_constructive} (with the same $k$), say 
$I_{h_1}$, is better than the approximation given by $h(\t)=0$---observe that 
0 is the constant that approximates $(\t,T-I_h(\t))$ the best. So, 
\begin{equation}\label{Fundamental}
\|T-(I_h(\t)+I_{h_1}(\t))\|_\infty
=\|(T-I_h(\t))+I_{h_1}(\t)\|_\infty
<\|T-I_h(\t)\|_\infty
\end{equation}
and we obtain that $I_h+I_{h_1}$ is a strictly better approximation to 
$T$ than $I_h$. 
%

This procedure of successively approximating the remainders is straightforward: we 
keep obtaining exponential approximations of the successive remainders as shown 
in Theorem \ref{t_exp_constructive} until we reach an alternated remainder. Then, 
this succesive sum of exponentials is the best approximation of $T$, as we will prove 
in the next result. For this next final result we are going to consider $T$ in 
case 1 and such that the solution of Problem \ref{problem2} exists.

\begin{theorem}\label{t_best_aprox_coincide}
The procedure of successively approximating the remainders locates, in a finite 
number of steps, the solution of Problem \rm{\ref{problem2}}.
\end{theorem}
\begin{proof}
Proposition \ref{c_M_m_unchanged} ensures that, until we find a remainder with 
alternating maxima and minima, the end of procedure, the remainder will be 
in case 1 after every iteration. Because of the way we construct every new 
exponential approximation, in each step we find, at least, either a new 
minimum before the previous ones (case $a=a_{im}$ for some $i<m$) or a new 
maximum after the previous ones (case $a=a_{jM}$ for some $j>M$). Anyway, 
in each step the instants where the first minimum and the last maximum 
are attained get closer and, in finite steps, obviously lower than $m-M+1$, the 
remainder becomes alternated, which means the procedure is over and that the 
result is optimal. 
%
%
\end{proof}

\begin{algorithm}[Successive approximation]\label{algorithm_remainders} 
{\tt Procedure successive\_remainders}\\ ($T=(T_1,\ldots,T_n)$ a vector of observations, 
$\t=(t_1,\ldots,t_n)$ the instants of time, $k < 0$ fixed value.) 
\begin{enumerate}
\item $a:=0$.
\item $b:=(\max(T)+\min(T))/2$.
\item $m:=\min \{ i \in \{ 1,\ldots,n \} \,:\, T_i = \min T\}$.
\item $M:=\max \{ i \in \{ 1,\ldots,n \} \,:\, T_i = \max T\}$.
\item \textbf{While} $M<m$:
  \begin{enumerate}[label={\arabic{enumi}.\arabic{enumii}}]
    \item\label{defi} $f:=ae^{kt}+b$.
    \item $x:=T-f$.
    \item $m_2=\min \{ j\in \{ 1,\ldots,m-1 \} \,:\, a_{jm} = 
    \min_{i=1,\ldots,m-1} \{ a_{im} = \frac{x_i-x_m}{e^{kt_i}-e^{kt_m}}  \} \}$.
    \item $M_2=\max \{ i\in \{ M+1,\ldots,n \} \,:\, a_{im} =  
    \min_{j=M+1,\ldots,n} \{ a_{jM} = \frac{x_j-x_M}{e^{kt_j}-e^{kt_M}}  \} \}$.
    \item $\bar{a}:=\min\{ a_{m_2 m}, a_{M_2 M} \}$.
    \item $a:=a+\bar{a}$.
    \item $b:=b+\frac{1}{2}(x_M-\bar{a}e^{kt_M}+x_m-\bar{a}e^{kt_m})$.
    \item \textbf{If} $a_{m_2 m} < a_{M_2 M}$ \textbf{then} $m:=m_2$, \textbf{else} $M:=M_2$.
  \end{enumerate}
\item Print($a,b$).
\end{enumerate}
\end{algorithm}

Please observe that this algorithm depends on the ability to determine whether 
$T$ is in case 1 or not---provided that all the maxima lie before the minima. 
Recall that Remez algorithm suffices for this, 
but also another version of Algorithm~\ref{algorithm_remainders} with $f:=at+b$ 
instead of the expression~\ref{defi} and taking into account that $a<0$. 

\section{Best approximation for Problem \ref{problem1}} \label{s_every_k}
Once we know how to solve Problem \ref{problem2}, which is explained in 
Section \ref{s_fixed_k}, we will solve Problem \ref{problem1}. Since the 
solution of Problem \ref{problem1} is also the solution of Problem \ref{problem2} 
for certain $k$, we just need to find the optimal $k$. The error 
function, $\EI (k)= \| T-f_k(\t)  \|_\infty$, where 
$f_k(\t)$ is the solution of Problem \ref{problem2} for every $k$, has a 
kind property: the quasiconvexity, as proven in \cite[theorem 2.17]{tac}. 
This property allows us to apply a variety of algorithms to approximate its 
minimum. The algorithm we present is a grid-search algorithm. We need to 
provide an interval where the optimal $k$ will be and, with every iteration, 
the algorithm will reduce the interval until it estimates the value that 
minimize the error function. In particular, the algorithm samples the error 
function in $d$ different points within the given interval and seek which 
one of them provides the minimum error. Then, it reduces the interval, 
centers it in the value that gives the minimum error and repeat the process 
until the interval is smaller than a stop condition.

\begin{algorithm}[Searching for the best $k$]\label{algorithm_best_approximation}
{\tt Procedure best\_approximation}\\ 
($T$ a vector of observations, 
$\t$ the instants of time, $[a_0,b_0]$ the interval where to 
seek the values of $k$, $d$ the number of samples, $f_k$ the best approximation 
of $T$ for $k$---obtained with~\ref{algorithm_remainders}---, 
$\varepsilon$ the stop condition.) 
\begin{enumerate}
\item \textbf{For every} $i \in \{1,\ldots,d\}$\textbf{:} 
$k_i := a_0+(i-1)\frac{b_0-a_0}{d-1}$,  $E_i := \| T-f_{k_i}(\t)  \|_\infty$.
\item $m := \min \{i\in\{1,\ldots,d\} \,:\, E_i = \min\{ E_1,\ldots,E_d \} \}$.
\item $l:=\frac{b_0-a_0}{d-1}$
\item \textbf{While} $l>\varepsilon$\textbf{:}
\begin{enumerate}[label={\arabic{enumi}.\arabic{enumii}}]
\item $a:=\max \{ k_m -l, a_0 \}$.
\item $b:=\min \{ k_m +l, b_0 \}$.
\item \textbf{For every} $i \in \{1,\ldots,d\}$; $k_i:=a+(i-1)\frac{b-a}{d-1}$,  
$E_i:= \| T-f_{k_i}(\t)  \|_\infty$.
\item $m := \min \{i\in\{1,\ldots,d\} \,:\, E_i = \min\{ E_1,\ldots,E_d \} \}$.
\item $l:=\frac{b-a}{d-1}$
\end{enumerate}
\item $k:=k_m$
\item Print($a,b,k$).
\end{enumerate}
\end{algorithm}

\section{Concluding remarks}

Algorithm~\ref{algorithm_best_approximation}, as described, will {\em always} 
converge to the solution. Namely, for each $k$ we obtain the exact values of 
$a$ and $b$ and the quasiconvexity of the error function assures that any 
{\em reasonable} algorithm will converge to the correct $k$. 

But there is another way to interpret what we have proved. Suppose that the 
best approximation $f(t)=ae^{k_0t}+b$ fulfils that there are exactly four 
indices where the equality 
\begin{equation}\label{eqn4}
|f(t_i)-T_i|=\|f(\t)-T\|_\infty
\end{equation}
holds, say $\{i_1,i_2,i_3,i_4\}$ is the set of {\em critical indices}. 
There must exist at least four indices, but 
there could be more---although that case is not easy to find since the usual 
datasets only contain rational numbers and the functions may take any real 
number. As shown in~\cite[corollary~1]{tac_mdpi}, the map that assigns to each 
$k$ the vector $f_k(\t)$ is continuous. In particular, it is clear that for 
some $\e>0$ and every $k\in(k_0-\e,k_0+\e)$ the three critical indices 
belong to the set $\{i_1,i_2,i_3,i_4\}$. 

As the reader can see in~\cite[lemma~3]{tac_mdpi}, the changes in the sets of 
critical indices happen at values of $k$ where there are four critical indices. 
So, take some $\cl{k}\in\R$ and $\e>0$. A moment's thought suffices to 
realize that, if the alternation of $T-f_{k}(\t)$ is 
\begin{equation}\label{los3.1}
f_k(t_{p_1})-T_{p_1}=-(f_k(t_{j_1})-T_{j_1})=f_k(t_{l_1})-T_{l_1}=\|f_k(\t)-T\|_\infty
\end{equation}
for every $k\in(\cl{k},\cl{k}+\e)$, then there is no way that any indices fulfil
\begin{equation}\label{los3.2}
-(f_{\cl{k}}(t_{p_2})-T_{p_2})=f_{\cl{k}}(t_{j_2})-T_{j_2}=
-(f_{\cl{k}}(t_{l_2})-T_{l_2})=\|f_{\cl{k}}(\t)-T\|_\infty
\end{equation}
unless in $\cl{k}$ we have the four alternated critical indices that ensure 
that $\cl{k}$ is optimal. So, when $k\to \cl{k}^+$ we have exactly three options: 
\begin{enumerate}
\item The alternance stays as $\max$-$\min$-$\max$. 
\item The alternance stays as $\min$-$\max$-$\min$. 
\item The value $\cl{k}$ is optimal. 
\end{enumerate}
So, the alternations in some $k_1$ and $k_2$ are different if and only if the 
optimal $k$ is between them.  
This allows us to determine whether the interval chosen at the beginning of 
Algorithm~\ref{algorithm_best_approximation} contains the optimal $k_0$ or not. 
Moreover, if $T$ is in case 1 and for some $k<0$ the alternance of $T-f_k(\t)$ 
is the same as that of $T-r(\t)$ (that is, $\max$-$\min$-$\max$), then the 
optimal $k_0$ belongs to $(-\infty,k)$. 

Ultimately, the goal of this approximation algorithm is easy to achieve if we 
know the indices where we need to look. Namely, if the critical indices 
are $\{i_1,i_2,i_3,i_4\}$, then the solution is the one that can be seen 
at the end of the proof of~\cite[lemma~2]{tac_mdpi}: $f(t)=b+a\exp(kt)$ with 
\begin{eqnarray}\label{los4}
a=&\frac{T_{i_1}-T_{i_3}}{\exp(kt_{i_1})-\exp(kt_{i_3})}\stackrel{*}=
\frac{T_{i_2}-T_{i_4}}{\exp(kt_{i_2})-\exp(kt_{i_4})}, \\
b=&\frac 12(T_{i_1}-a\exp(kt_{i_1})+T_{i_2}-a\exp(kt_{i_2})), \nonumber
\end{eqnarray}
observe that the {\em starred} equality $\stackrel{*}=$ determines $k$. 

So, in almost any case, we may find analytically the solution of Problem~\ref{problem1}. 
However, we have avoided to include this in the previous sections because there 
is no {\em a priori} way to ensure that $k_1$ and $k_2$ are close enough. In 
spite of this, if the approximation given by~\cite[lemma~2]{tac_mdpi} for some 
$\{i_1,i_2,i_3,i_4\}$, say $f(\t)$, leaves every other $(t_i,T_i)$ closer to 
$f(t_i)$, then of course it is the best approximation. 

Another reason to avoid this complement to the algorithm is that we have 
actually been able to approximate some datasets with more precision with 
Algorithm~\ref{algorithm_best_approximation}, probably because of the rounding 
errors. In cases when there is a huge need for precision it could be better 
to add this to the original algorithm, though. 

\section{Acknowledgements}
The second author is partially supported by Project MINCIN 
PID2019-103961GB-C21 (Spain) and Junta de Extremadura, Project IB20038. 


\bibliographystyle{plain}

\bibliography{Cheby}

\begin{thebibliography}{1}

\bibitem{tac_mdpi}
Javier Cabello~S\'anchez, Juan~Antonio Fern\'andez~Torvisco, and Mariano
  Rodr\'iguez-Arias~Fern\'andez.
\newblock Tac method for fitting exponential autoregressive models and others:
  Applications in economy and finance.
\newblock {\em Mathematics}, 9(8), 2021.

\bibitem{Descloux}
Jean Descloux.
\newblock Approximations in ${L}^p$ and {C}hebyshev approximations.
\newblock {\em Journal of the Society for Industrial and Applied Mathematics},
  11(4):1017--1026, 1963.

\bibitem{tac}
J.~A. Fern\'andez~Torvisco, M.~Rodr\'iguez-Arias~Fern\'andez, and
  J.~Cabello~S\'anchez.
\newblock A new algorithm to fit exponential decays without initial guess.
\newblock {\em Filomat}, 32:4233--4248, 01 2018.

\bibitem{Quesada}
Jos{\'e}~M Quesada, J~Fern{\'a}ndez-Ochoa, Juan Mart{\'\i}nez-Moreno, and Jorge
  Bustamante.
\newblock The {P}olya algorithm in sequence spaces.
\newblock {\em Journal of Approximation Theory}, 135(2):245--257, 2005.

\bibitem{ReemtsenRemez}
Rembert Reemtsen.
\newblock Modifications of the first {R}emez algorithm.
\newblock {\em SIAM Journal on Numerical Analysis}, 27(2):507--518, 1990.

\bibitem{Remez}
E.Y. Remez.
\newblock General computational methods for {T}chebycheff approximation.
\newblock {\em Atomic Energy Commission Translations}, 4491:1--85, 1957.

\bibitem{Saramaki}
Tapio Saram{\"a}ki and Yong~Ching Lim.
\newblock Use of the {R}emez algorithm for designing frm based firr filters.
\newblock {\em Circuits, Systems and Signal Processing}, 22:77--97, 2003.

\end{thebibliography}

\end{document}